\documentclass[11pt, a4paper, notitlepage]{article}

\usepackage{amsmath, latexsym, amsthm, amsfonts} 
\usepackage{graphicx} 
\usepackage{epsfig}
\usepackage{varioref}
\usepackage{natbib} 

\bibpunct{(}{)}{;}{a}{}{,} 

\theoremstyle{plain}
\newtheorem{thm}{Theorem}[section]
\newtheorem{prop}{Proposition}[section]

\newtheorem{lem}{Lemma}[section]

\newtheorem{cnd}{Condition}[section]

\theoremstyle{definition} 

\newtheorem{rem}{Remark}[section]

\newcommand{\infint}{\int_{-\infty}^{\infty}}
\newcommand{\Log}{\operatorname{Log}}
\def\convd{\stackrel{\cal D}{\rightarrow}}
\def\convp{\stackrel{\operatorname{P}}{\rightarrow}}
\def\ex{{\rm E\,}}
\def\var{\mathop{\rm Var}\nolimits}

\begin{document}
\date{\today}
\title{Decompounding under Gaussian noise}
\author{Shota Gugushvili\thanks{The research of the author was financially supported by the
Nederlandse Organisatie voor Wetenschappelijk Onderzoek (NWO).}\\
{\normalsize Korteweg-de Vries Institute for Mathematics}\\
{\normalsize Universiteit van Amsterdam}\\
{\normalsize Plantage Muidergracht 24}\\
{\normalsize 1018 TV Amsterdam}\\
{\normalsize The Netherlands}\\
{\normalsize sgugushv@science.uva.nl}}
\maketitle
\begin{abstract}
Assuming that a stochastic process\index{process!L\'evy} $X=(X_t)_{t\geq 0}$ is a sum of a compound Poisson process
$Y=(Y_t)_{t\geq 0}$ with known intensity $\lambda$ and unknown jump size density $f,$ and an independent
Brownian\index{motion!Brownian} motion $Z=(Z_t)_{t\geq 0},$ we consider the problem of nonparametric estimation of $f$
from low frequency observations from $X.$ The estimator of $f$ is constructed via Fourier inversion and kernel
smoothing. Our main result deals with asymptotic normality of the proposed estimator at a fixed point.
\end{abstract}
\newpage

\section{Introduction}

\label{deconvnoise-intro} Let $Y=(Y_t)_{t\geq0}$ be a compound Poisson\index{process!compound Poisson} process with
intensity $\lambda$ and jump size distribution $F,$ which has a density $f.$ Assume that $Z=(Z_t)_{t\geq0}$ is a
Brownian\index{motion!Brownian} motion independent of $Y$ and consider the stochastic process $X_t=Y_t+Z_t.$ Notice
that $X=(X_t)_{t\geq 0}$ is a L\'evy process\index{process!L\'evy}. Suppose that $X$ is observed at equidistant time
points $\Delta,2\Delta,\ldots,n\Delta.$ By a rescaling argument, without loss of generality, we may take $\Delta=1.$
Given a sample $X_1,X_2,\ldots,X_n,$ the statistical problem we consider is nonparametric estimation of the density
$f.$ Notice that the L\'evy\index{triplet!L\'evy} triplet of the process $X$ is given by $(0,1,\nu),$ where the L\'evy
measure $\nu(dx)$ equals $\lambda f(x)dx,$ see \citet[Example~8.5]{sato}. Since the L\'evy triplet provides a unique
means for characterisation of any L\'evy process\index{process!L\'evy}, see e.g.\ \citet[Chapter~2]{sato}, inference on
the law of $X$ can be reduced to inference on $\nu.$ Most of the existing literature dealing with estimation problems
for L\'evy processes\index{process!L\'evy} is concerned with parametric estimation of the L\'evy measure, see e.g.\
\citet{akritas1} and \citet{akritas2}, where a fairly general setting is considered. There are relatively few papers
that study nonparametric inference procedures for L\'evy processes\index{process!L\'evy}, and the majority of them
assume that high frequency data\index{data!high frequency} are available, i.e.\ either a L\'evy
process\index{process!L\'evy} is observed continuously over a time interval $[0,T]$ with $T\rightarrow\infty,$ or it is
observed at equidistant time points $\Delta_n,\ldots,n\Delta_n$ and $\lim_{n\rightarrow \infty}\Delta_n=0,$
$\lim_{n\rightarrow\infty}n\Delta_n=\infty,$ see e.g.\ \citet{rubin}, \citet{basawa3} and \citet{figueroa}.
On the other hand, high frequency data\index{data!high frequency} are not always available and it is interesting to
study estimation problems for this case as well. In the particular context of a compound Poisson\index{process!compound
Poisson} process we mention \citet{bu,bugr} and \citet{gug}, where given a sample $Y_1,\ldots,Y_n$ from a compound
Poisson\index{process!compound Poisson} process $Y=(Y_t)_{t\geq 0},$ nonparametric estimators for the jump size
distribution function $F$ (see \citet{bu,bugr}) and its density $f$ (see \citet{gug}) are proposed and their
asymptotics are studied as $n\rightarrow \infty.$ This problem is referred to as decompounding. The process
$X_t=Y_t+Z_t$ constitutes a generalisation of the compound Poisson model considered in \citet{bu,bugr} and \citet{gug}
and is related to Merton's jump-diffusion model of an asset price, see \citet{merton}. Since $Z$ is a
Brownian\index{motion!Brownian} motion, it is natural to call the estimation problem of $f$ decompounding under
Gaussian noise. Figures \ref{BMPath}--\ref{BMwithjumpsdiscretised} provide an indication of the difficulty of the
problem. Figure \vref{BMPath} gives a typical path of the Brownian\index{motion!Brownian} motion, while Figure
\ref{BMwithjumps} gives a path of the process $X.$
\begin{figure}[htb]
\setlength{\unitlength}{1cm}
\begin{minipage}{6cm}
\begin{picture}(5.5,4.0)
\epsfxsize=5.5cm\epsfysize=4cm\epsfbox{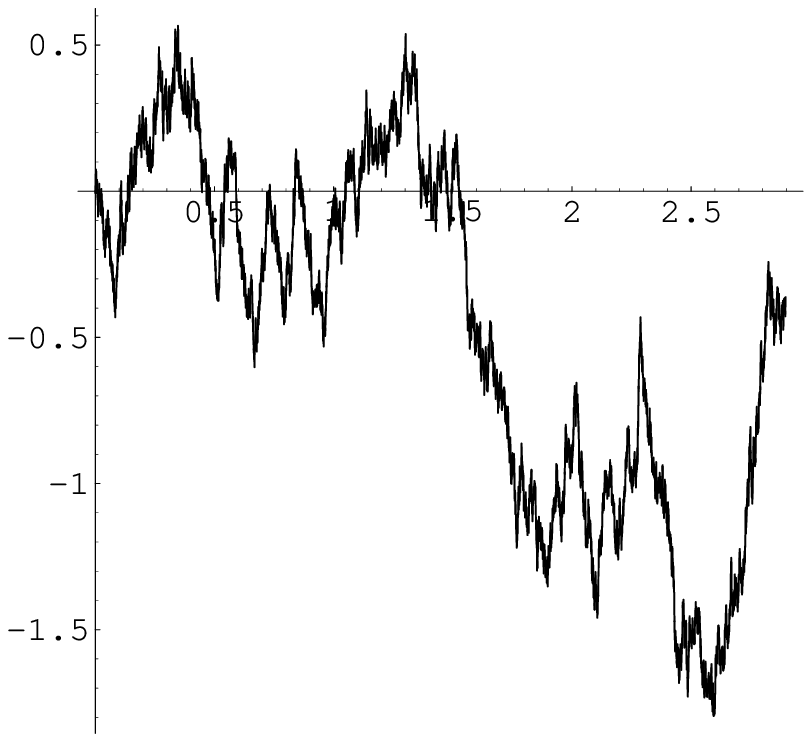}
\end{picture}
\par
\caption{\label{BMPath} A typical path of the Brownian motion.}
\end{minipage}
\begin{minipage}{6cm}
\begin{picture}(5.5,4.0)
\epsfxsize=5.5cm\epsfysize=4cm\epsfbox{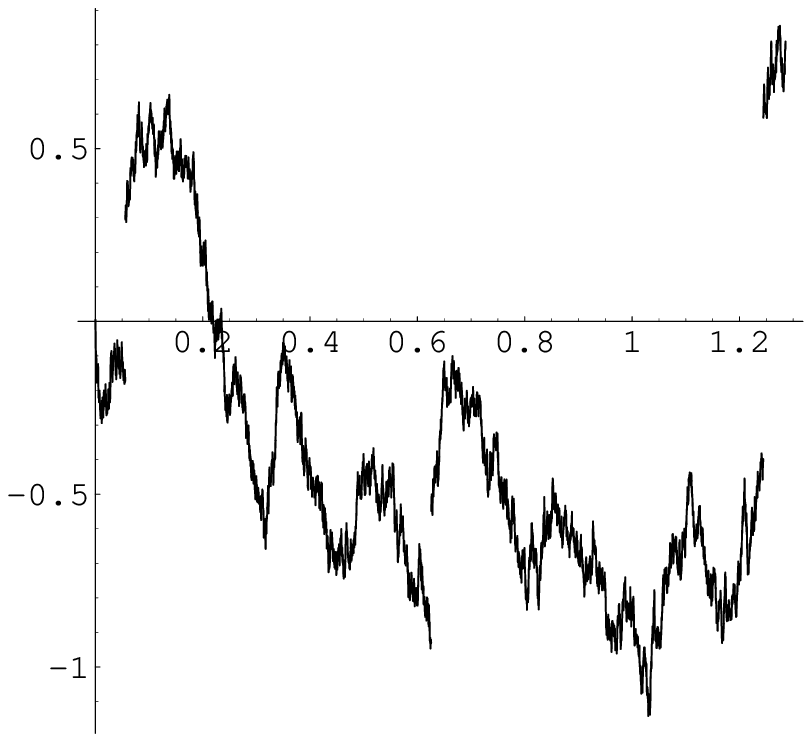}
\end{picture}
\par
\caption{\label{BMwithjumps} A typical path of the process $X=(X_t)_{t\geq 0}.$}
\end{minipage}
\end{figure}
The difference is at once clear when $X$ is observed continuously. If this is the case, then one can see all the jumps
in the path of $X$ and the problem of estimating $f$ is relatively easy, as no decompounding is involved. On the other
hand Figures \ref{BMPathdiscretised} and \vref{BMwithjumpsdiscretised} provide discretised versions of the typical
paths of the Brownian\index{motion!Brownian} motion $Z$ and the process $X.$
\begin{figure}[htb]
\setlength{\unitlength}{1cm}
\begin{minipage}{6cm}
\begin{picture}(5.5,4.0)
\epsfxsize=5.5cm\epsfysize=4cm\epsfbox{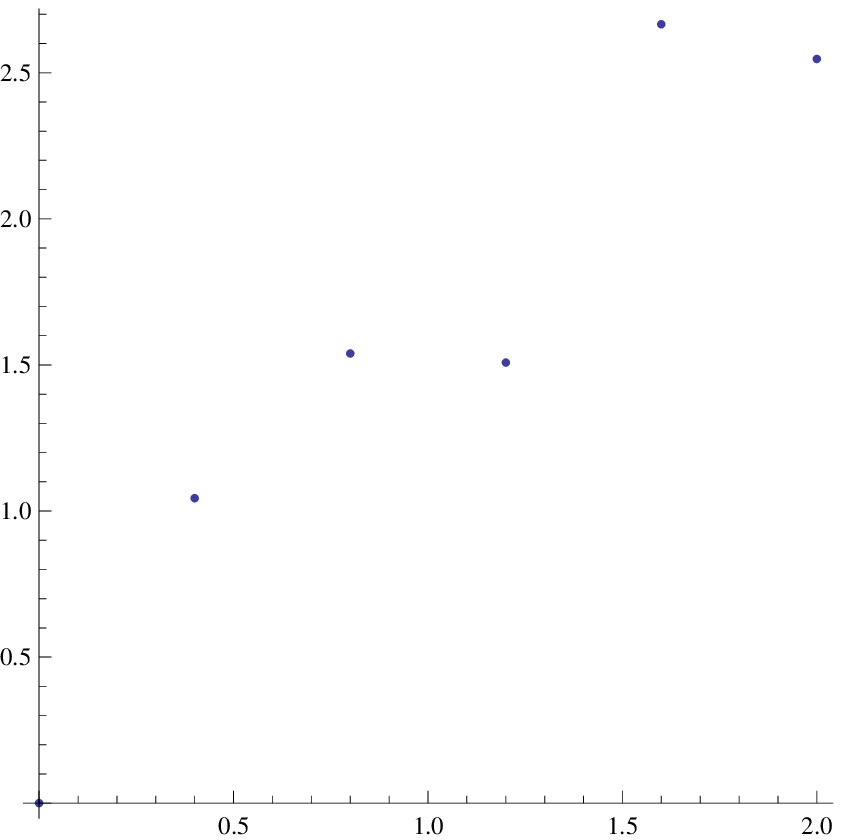}
\end{picture}
\par
\caption{\label{BMPathdiscretised} A discretised path of the Brownian motion.}
\end{minipage}
\begin{minipage}{6cm}
\begin{picture}(5.5,4.0)
\epsfxsize=5.5cm\epsfysize=4cm\epsfbox{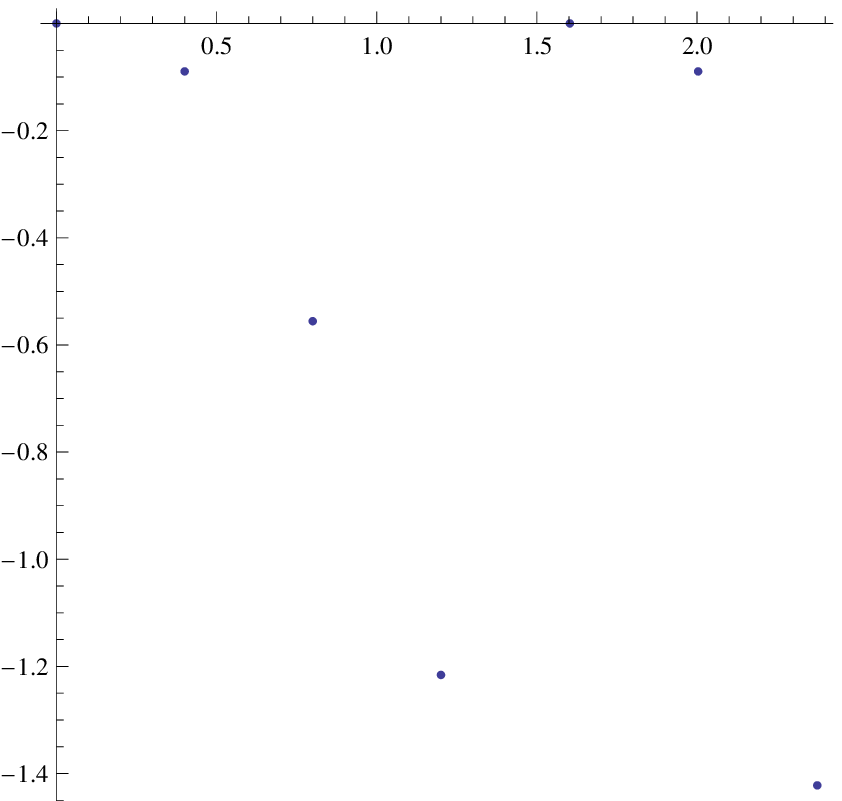}
\end{picture}
\par
\caption{\label{BMwithjumpsdiscretised} A discretised path of the process $X=(X_t)_{t\geq 0}.$}
\end{minipage}
\end{figure}
In this case both plots look similar and given the highly irregular character of Brownian paths, it is difficult to
conclude at which time instances jumps occur in the process $X$. The information on $f$ is contained in the jumps and
the impossibility to observe them makes the problem of estimation of $f$ much more difficult.

Nonparametric estimation of the L\'evy measure of a more general L\'evy process\index{process!L\'evy} than $X$ based on
low frequency observations\index{observations!low frequency} was considered in \citet{watteel} and \citet{reiss}.
However these authors treat the case of estimation of the L\'evy measure only (or of the canonical function $K$ in case
of \citet{watteel}) and not of its density. Moreover, they study the proposed estimators under the strong moment
condition $\ex[|X_1|^{4+\delta}]<\infty,$ where $\delta$ is some strictly positive number. This condition automatically
excludes distributions with heavy tails. We refer to those papers for additional details.

Using the stationary independent increments property of a L\'evy process\index{process!L\'evy}, we see that the problem
of estimating $f$ from a discrete time sample from $X$ is equivalent to the following: let $X_1,\ldots, X_n$ be i.i.d.\
observations\index{observations!i.i.d.}, where $X_i=Y_i+Z_i,$ and $Y_i$ and $Z_i$ are independent. Assume that the
unobservable $Y$'s are distributed as a random variable
\begin{equation*}
Y=\sum_{j=1}^{N(\lambda)}W_j,
\end{equation*}
where $N(\lambda)$ has a Poisson distribution with parameter $\lambda$ and where the $W$'s are i.i.d.\ with
distribution function $F$ and density $f$ and where by convention a sum over the empty set is understood to be zero.
Thus we assume that $Y$ is a Poisson sum of i.i.d.\ $W$'s. Furthermore, let the random variables $Z_i$ have a
standard\index{distribution!standard normal} normal distribution. Assume that $\lambda$ is known. The estimation
problem is as follows: based on the sample $X_1,\ldots ,X_n,$ construct an estimator of $f.$

In this context one might also think of the $X$'s as of measurements of the realisations $Y$'s of some quantity of
interest, which are corrupted by the noise $Z.$ This way we are in the classical 'signal' plus 'noise' setting and the
problem at hand is then related to the deconvolution\index{deconvolution} problem, see e.g.\ \citet{jones} for an
overview, and in particular to its generalisation to the case of an atomic\index{deconvolution!atomic} deconvolution,
see \citet{gug2}.

The method that will be used to construct an estimator for $f$ is based on Fourier inversion and is similar in spirit
to the use of kernel estimators in deconvolution\index{deconvolution} problems, as well as our approach in
\citet{gug,gug2}.

Let $\phi_X,$ $\phi_Y,$ $\phi_Z$ and $\phi_f$ denote the characteristic functions of $X,Y,Z,$ and $W,$ respectively.
Then by independence of $Y$ and $Z$ and by the fact that
\begin{equation*}
\phi_Y(t)=e^{-\lambda+\lambda\phi_f(t)},
\end{equation*}
see e.g.\ \citet[Chapter~1,~Section~4]{sato}, we have
\begin{equation}
\label{deconv-phix} \phi_X(t)=\phi_Y(t)\phi_Z(t)=e^{-\lambda+\lambda\phi_f(t)}e^{-t^2/2},
\end{equation}
and therefore
\begin{equation}
\label{deconv-inv} e^{\lambda\phi_f(t)}=\frac{\phi_X(t)}{e^{-\lambda}e^{-t^2/2}}.
\end{equation}
Notice that $P(Y=0)=e^{-\lambda}.$ Inverting \eqref{deconv-inv}, we get
\begin{equation*}
\phi_f(t)=\frac{1}{\lambda}\operatorname{Log}\left(\frac{\phi_X(t)}{e^{-\lambda}e^{-t^2/2}}\right).
\end{equation*}
Here $\operatorname{Log}$ denotes the \textit{distinguished logarithm}\index{logarithm!distinguished}, called so due to
the similarity to the distinguished logarithm\index{logarithm!distinguished} as constructed e.g.\ in
\citet[Lemma~1,~p.~413]{chow}, \citet[Theorem~7.6.2]{chung}, \citet{fink} and \citet[Lemma~7.6]{sato}. The difference
is that in our case the function $\operatorname{exp}(\lambda\phi_f(t))$ equals $e^{\lambda}$ at $t=0$ and not $1.$ The
distinguished logarithm of $\operatorname{exp}(\lambda\phi_f(t))$ in our case can be {\it defined} as
\begin{equation}
\label{distlog} \lambda+\operatorname{Log}\left(\frac{\phi_X(t)}{e^{-t^2/2}}\right),
\end{equation}
where $\Log$ denotes the distinguished logarithm as constructed e.g.\ in \citet[Theorem~7.6.2]{chung}, or it can be
constructed directly.
\begin{rem}
Notice that in general the distinguished logarithm\index{logarithm!distinguished} of the non-vanishing characteristic
function $\phi(t)$ cannot be reduced to the composition of the principal branch\index{logarithm!principal branch of} of
an ordinary logarithm $\log$ with $\phi.$ Consider the following trivial example: $\phi(t)=e^{it}.$ This characteristic
function satisfies the requirements of \citet[Theorem~7.6.2]{chung}, since it takes its values on the unit circle in
the complex plane and hence its distinguished logarithm\index{logarithm!distinguished} exists and is given by
$\Log(\phi(t))=it.$ On the other hand if one considers the argument of $\log(\phi(t)),$ it is easy to see that it jumps
whenever $\phi$ crosses the negative real axis, see Figure \vref{circle_argument} and compare to the argument of the
distinguished logarithm\index{logarithm!distinguished}. This fact is not surprising, given that $-1$ lies on the branch
cut of the principal branch\index{logarithm!principal branch of} of an ordinary logarithm.
\end{rem}
\begin{figure}[htb]
\setlength{\unitlength}{1cm}
\begin{center}
\epsfxsize=5.5cm\epsfysize=5.5cm\epsfbox{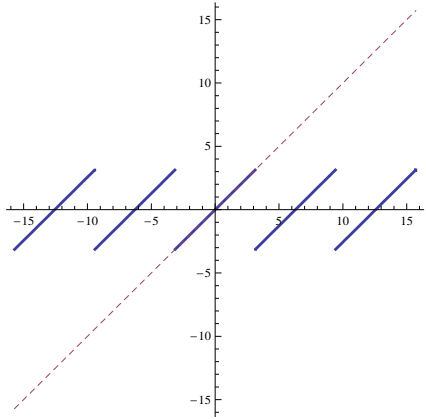}
\par
\caption{\label{circle_argument} Arguments of the principal branch of a logarithm and of the distinguished logarithm.}
\end{center}
\end{figure}
\begin{rem}
Notice that if $\lambda<\log 2,$ the distinguished logarithm in \eqref{distlog} reduces to a composition of the
principal branch of an ordinary logarithm with $\operatorname{exp}(\lambda\phi_f(t)).$ This follows from the fact that
\begin{equation*}
\Log\left(e^{\lambda\phi_f(t)}\right)=\Log\left((e^{\lambda}-1)\phi_g(t)+1\right),
\end{equation*}
where $\phi_g(t)=\phi_{Y|N>0}(t),$ cf.\ \citet{gug}. It is immediately seen that the condition $\lambda<\log2$ will
then prevent $\exp(\lambda\phi_f(t))$ from taking values on the negative real axis, which constitutes the branch cut
for the principal branch of an ordinary logarithm.
\end{rem}

Assuming that $\phi_f$ is integrable, by Fourier inversion we obtain
\begin{equation*}
f(x)=\frac{1}{2\pi\lambda}\infint e^{-itx}\operatorname{Log}\left(\frac{\phi_X(t)}{e^{-\lambda}e^{-t^2/2}}\right)dt.
\end{equation*}
This expression will be used as the basis for construction of an estimator of $f.$ Let $\phi_{emp}$ denote the
empirical characteristic function\index{empirical characteristic function} of the sample $X_1,\ldots,X_n,$
\begin{equation*}
\phi_{emp}(t)=\frac{1}{n}\sum_{j=1}^{n}e^{itX_j}.
\end{equation*}
Furthermore, let $w$ be a symmetric kernel\index{kernel} with Fourier transform $\phi_w$ supported on $[-1,1]$ and
nonzero there, and let $h>0$ be a bandwidth\index{bandwidth}. The density $q$ of $X$ can then be estimated by a kernel
density estimator
\begin{equation*}
q_{nh}(x)=\frac{1}{nh}\sum_{j=1}^n w\left(\frac{x-X_j}{h}\right).
\end{equation*}
Its characteristic function $\phi_{q_{nh}}(t)=\phi_{emp}(t)\phi_w(ht)$ will serve as an estimator of $\phi_X(t).$ For
those $\omega$'s from the sample space $\Omega,$ for which the distinguished logarithm\index{logarithm!distinguished}
in the integral below is well-defined, $f$ can be estimated by the following plug-in type estimator,
\begin{equation}
\label{deconvnoise-fnh} f_{nh}(x)=\frac{1}{2\pi\lambda}\int_{-1/h}^{1/h}
e^{-itx}\operatorname{Log}\left(\frac{\phi_{emp}(t)\phi_w(ht)}{e^{-\lambda}e^{-t^2/2}}\right)dt,
\end{equation}
while for those $\omega$'s, for which the distinguished logarithm\index{logarithm!distinguished} cannot be defined, we
can assign an arbitrary value to $f_{nh}(x),$ e.g.\ zero. The distinguished logarithm\index{logarithm!distinguished} in
\eqref{deconvnoise-fnh} can be defined only for those $\omega $'s for which
$\phi_{emp}(t)\phi_w(ht)e^{\lambda}e^{t^2/2}$ as a function of $t$ does not vanish on $[-1/h,1/h].$ In fact in Section
\ref{deconvnoise-results} we will prove that as $n\rightarrow\infty,$ the probability of the exceptional set where the
distinguished logarithm\index{logarithm!distinguished} is undefined, tends to zero. For technical reasons which will
become apparent in the proofs, we also need to truncate $f_{nh}(x),$ and consequently, we define the estimator of
$f(x)$ not by the expression above, but by
\begin{equation}
\label{deconvnoise-hatfnh} \hat{f}_{nh}(x)=(M_n\wedge f_{nh}(x))\vee(-M_n),
\end{equation}
where $M=(M_n)_{n\geq1}$ denotes a sequence of positive numbers converging to infinity at a suitable rate to be
specified below.

Concluding this section, we state conditions on the density $f,$ the bandwidth\index{bandwidth} $h$ and the truncating
sequence $M$ that will be used in Section \ref{deconvnoise-results}.
\begin{cnd}
\label{conditionf} Let the density $f$ be such that $\phi_f$ is integrable.
\end{cnd}
\begin{cnd}
\label{conditionw} Let the kernel\index{kernel} $w$ be the sinc kernel\index{kernel!sinc}, $w(x)=(\sin x)/\pi x.$
\end{cnd}
The Fourier transform of the sinc kernel\index{kernel!sinc} is given by $\phi_w(t)=1_{[-1,1]}(t).$
The sinc kernel\index{kernel!sinc} has been used successfully in kernel density estimation since a long time, see e.g.\
\citet{davis1,davis2}. It is the simplest example among the so-called superkernels\index{superkernel}, i.e.\ kernels
the Fourier transforms of which are identically $1$ in some open neighbourhood of zero. For more information on the
latter class of kernels\index{kernel} we refer e.g.\ to \citet{dev}, \citet{devr} or \citet{devroye}. An attractive
feature of the sinc kernel\index{kernel!sinc} in ordinary kernel density estimation is that it is asymptotically
optimal when one selects the mean square\index{error!mean square} error or the mean integrated square
error\index{error!mean integrated square (MISE)} as the criterion for the performance of an estimator. Notice that the
sinc kernel\index{kernel!sinc} is not Lebesgue integrable, but its square is.
\begin{cnd}
\label{conditionh} Let the bandwidth\index{bandwidth} $h_n$ depend on $n$ and be such that $h_n\sim (\log n)^{-\beta},$
where $\beta<1/2.$
\end{cnd}
Notice that this condition implies $ne^{-1/h_n^2}\rightarrow\infty.$ In the sequel we will suppress the subscript used
to demonstrate the dependence of $h$ on $n,$ since no ambiguity will arise.
\begin{cnd}
\label{conditionm} Let the truncating sequence $M=(M_n)_{n\geq 1}$ be such that $M_n=C\log n,$ where $C>0$ is some
constant.
\end{cnd}

The rest of the paper is organised as follows: in Section \ref{deconvnoise-results} we show that with probability
approaching $1$ as $n\rightarrow\infty,$ the distinguished logarithm\index{logarithm!distinguished} in
\eqref{deconvnoise-fnh} is well-defined and subsequently we state the main result of the paper concerning the
asymptotic normality\index{asymptotic normality} of $\hat{f}_{nh}$ at a fixed point $x.$ The section is concluded with
a brief discussion of the obtained results. Section \ref{deconvnoise-simulation} contains a simulation example.
All the proofs are collected in Section \ref{deconvnoise-proofs}.

\section{Main result}
\label{deconvnoise-results} We first establish that with probability tending to $1$ as $n\rightarrow\infty,$ the
distinguished logarithm\index{logarithm!distinguished} in \eqref{deconvnoise-fnh} is well-defined. Thus our goal is to
find a set $B_{nh},$ such that on this set the distinguished logarithm\index{logarithm!distinguished} might be
undefined, while on the set $B_{nh}^c$ it is well-defined. Fix $\omega$ from the sample space $\Omega$ and consider the
quantity
\begin{equation}
\label{deconvexpb1} \sup_{t\in\ \left [-\frac{1}{h},\frac{1}{h}\right
]}\left|\frac{\phi_{emp}(t)}{e^{-\lambda}e^{-t^2/2}}-\frac{\phi_{X}(t)}{e^{-\lambda}e^{-t^2/2}}\right|.
\end{equation}
Now suppose that there exists a small number $\delta,$ such that
\begin{equation*}
\sup_{t\in\ \left [-\frac{1}{h},\frac{1}{h}\right
]}e^{1/(2h^2)}\left|\frac{\phi_{emp}(t)}{e^{-\lambda}}-\frac{\phi_{X}(t)}{e^{-\lambda}}\right|\leq\delta.
\end{equation*}
Obviously this implies that \eqref{deconvexpb1} is less than $\delta$. If $\delta$ is small enough, then since
$\phi_X(t)e^{\lambda}e^{t^2/2}=\operatorname{exp}[\lambda\phi_f(t)]$ is bounded away from zero, also
$\phi_{emp}(t)e^{\lambda}e^{t^2/2}$ will be bounded away from zero on $[-1/h,1/h].$ From this it follows that on this
interval one can define the distinguished logarithm\index{logarithm!distinguished} of
$\phi_{emp}(t)e^{\lambda}e^{t^2/2}.$ This simple observation shows that on the set
\begin{equation*}
B_{nh}^{c}=\left\{\omega: \sup_{t\in\ \left [-\frac{1}{h},\frac{1}{h}\right
]}e^{{1}/{(2h^2)}}\left|\frac{\phi_{emp}(t)}{e^{-\lambda}}-\frac{\phi_{X}(t)}{e^{-\lambda}}\right|\leq\delta\right\}
\end{equation*}
the distinguished logarithm\index{logarithm!distinguished} will be well-defined for $\delta$ sufficiently small. Thus,
what remains to be done is to prove that the probability of the complement of this set converges to zero as
$n\rightarrow\infty.$ To this end we will make use of the following theorem from \citet{devroye-chiu}.
\begin{thm}
\label{devroye} Let $X$ be a random variable with characteristic function $\phi$ and finite first moment, and let
$\phi_n$ be the empirical characteristic function\index{empirical characteristic function} of the i.i.d.\ sample
$X_1,\ldots,X_n$ drawn from $X.$ Then, for $\alpha$ and $\beta,$ possibly dependent upon $n,$
\begin{equation}
\label{st} \operatorname{P}\left(\sup_{|t|<\alpha} |\phi_n(t)-\phi(t)|>\beta\right)\leq 4
\left(1+\frac{8\alpha\ex[|X|]}{\beta}\right)e^{-n\beta^2/72}+o(1),
\end{equation}
where the $o(1)$ term is uniform over all $\alpha$ and $\beta.$
\end{thm}
\begin{rem} In our results we need additional information on the $o(1)$ term in \eqref{st}. It follows from the proof
of Theorem \ref{devroye} that it is bounded by
\begin{equation}
\label{o1term}
P\left(\left|\frac{1}{n}\sum_{j=1}^nX_j\right|\geq\frac{4}{3}\ex[|X_1|]\right),
\end{equation}
see \citet{devroye-chiu}. Since the $X$'s are not bounded, it is not possible to apply Hoeffding's
inequality\index{inequality!Hoeffding's}, see \citet{hoeffding}, to show that this probability is exponentially small.
At the same time, verification of the moment conditions needed for Bernstein's inequality\index{inequality!Bernstein's}
to hold is difficult in our case and might require strong conditions on $Y.$ Therefore we opt for an unsophisticated
application of Chebyshev's inequality\index{inequality!Chebyshev's} to bound this probability.
\end{rem}

The following proposition follows from Theorem \ref{devroye}.
\begin{prop}
\label{deconvnoise-log} Assume Conditions \ref{conditionw} and \ref{conditionh} and let $\ex[|X|]<\infty.$ Then the distinguished
logarithm\index{logarithm!distinguished} in \eqref{deconvnoise-fnh} is well-defined with probability tending to $1$ as
$n\rightarrow\infty.$ Moreover, if $\ex[|X|^{\rho}]<\infty$ for $1<\rho<2,$ then
\begin{equation*}
\operatorname{P}(B_{nh})=O\left(\frac{1}{n^{\rho-1}}\right),
\end{equation*}
and if $\ex[|X|^{\rho}]<\infty$ for $\rho\geq2,$ then
\begin{equation*}
\operatorname{P}(B_{nh})=O\left(\frac{1}{n^{\rho/2}}\right).
\end{equation*}
\end{prop}

The main result of the paper concerns the asymptotic normality\index{asymptotic normality} of $\hat{f}_{nh}(x)$ at a
fixed point $x.$ The following theorem holds true.
\begin{thm}
\label{deconvnoise-an1} Suppose that $\lambda$ is known. Let the estimator $\hat{f}_{nh}(x)$ be defined as in
\eqref{deconvnoise-hatfnh}, and assume that Conditions \ref{conditionf}--\ref{conditionm} hold. Furthermore, let
$|x|^{\rho}f(x)$ be integrable with $\rho>3/2.$ Then
\begin{equation*}
\frac{\sqrt{n}}{he^{{1}/{(2h^2)}}}(\hat{f}_{nh}(x)-\operatorname{E}[\hat{f}_{nh}(x)])\convd
{\mathcal{N}}\left(0,\frac{e^{2\lambda}}{2\pi^2\lambda^2}\right)
\end{equation*}
as $n\rightarrow\infty.$
\end{thm}
\begin{rem}
Notice that the integrability of $|x|^{\rho}f(x)$ implies $\ex[|X|^{\rho}]<\infty,$ see e.g.\
\citet[Corollary~25.8]{sato}. Thus the conditions of the Theorem \ref{deconvnoise-an1} cover a large class of
distributions with heavy tails.
\end{rem}
\begin{rem}
From Theorem \ref{deconvnoise-an1} it follows, that in order to get a consistent estimator, $nh^{-1}e^{-1/h^2}$ has to
diverge to infinity. This means that the bandwidth\index{bandwidth} $h$ has to be fairly large, i.e.\ of order $(\log
n)^{-\beta},$ where $\beta\leq 1/2,$ thus resulting in a slow, logarithmic rate of convergence\index{convergence
rate!logarithmic} of $\hat{f}_{nh}(x).$ This is in sharp contrast with the ordinary decompounding case, where the
convergence rate is polynomial\index{convergence rate!polynomial}, see Section \citet{gug}. On the other hand, the
convergence rate\index{convergence rate} of $\hat{f}_{nh}(x)$ is similar to that in the ordinary
deconvolution\index{deconvolution}, as well as the deconvolution for an atomic distribution\index{distribution!atomic},
when the error distribution is assumed to be supersmooth, see e.g.\ \citet{fan1} and \citet{gug2}. This fact should not
come as a surprise, due to the similar structure of these problems and the presence of Gaussian noise in our model. We
also mention that in a recent preprint \citet{reiss}, under some conditions on the L\'evy measure $\nu,$ obtained
similar logarithmic lower bounds for estimation from low frequency\index{observations!low frequency} observations of
the L\'evy measure $\nu$ of a general L\'evy process\index{process!L\'evy} with a Brownian component.
\end{rem}
\begin{rem}
Using the estimator $p_{ng}$ from \citet{gug2}, an estimator of $\lambda$ can be defined as $\lambda_{ng}=-\log
p_{ng},$ of course provided that $p_{ng}$ is strictly positive. However the proof of Theorem \ref{deconvnoise-an1} for
the case of unknown $\lambda$ is a highly nontrivial task.
\end{rem}

Apart of Theorem \ref{deconvnoise-an1}, it is also interesting to study the asymptotic distribution of
\begin{equation}
\label{deconvnoise-fnhf} \frac{\sqrt{n}}{he^{{1}/{(2h^2)}}}(\hat{f}_{nh}(x)-f(x)),
\end{equation}
i.e.\ of the estimator $\hat{f}_{nh}(x)$ centred at the true density $f(x).$ After rewriting the above expression as
\begin{equation}
\begin{split}
\label{deconvnoise-anf}
\frac{\sqrt{n}}{he^{{1}/{(2h^2)}}}(\hat{f}_{nh}(x)-f(x))&=\frac{\sqrt{n}}{he^{{1}/{(2h^2)}}}(\hat{f}_{nh}(x)-\ex[\hat{f}_{nh}(x)])\\
&+\frac{\sqrt{n}}{he^{{1}/{(2h^2)}}}(\ex[\hat{f}_{nh}(x)]-f(x)),
\end{split}
\end{equation}
we see that we have to study the behaviour of the bias\index{bias} of the estimator $\hat{f}_{nh}(x),$ which is given
by $\ex[\hat{f}_{nh}(x)]-f(x).$ It will turn out that the behaviour of the bias\index{bias} depends on the tail
behaviour of the characteristic function of $f.$ For our purposes it suffices to distinguish two cases: in the first case we will assume that
$\phi_f(t)=O(e^{-|t|^{\alpha}})$ with $1<\alpha\leq2,$ and in the second case we will assume that
$\phi_f(t)=O(|t|^{-\gamma})$ as $t\rightarrow\infty$ with $\gamma>1.$ These two cases find a parallel in
deconvolution\index{deconvolution} problems, where a distinction is made between the use of supersmooth or ordinary
smooth distributions to model the error distribution, see e.g. \citet{fan1}.
\begin{prop}
\label{deconvnoise-bias} Suppose that $\lambda$ is known. Let the estimator $\hat{f}_{nh}(x)$ be defined as in
\eqref{deconvnoise-hatfnh}, and assume that Conditions \ref{conditionf}--\ref{conditionm} hold. Furthermore, let $f$
have a finite $\rho$th moment, $\rho>1.$

(i) If $\phi_f(t)=O(e^{-|t|^{\alpha}})$ as $|t|\rightarrow\infty$ for $1<\alpha\leq 2,$ then we have
\begin{equation*}
\ex[\hat{f}_{nh}(x)]-f(x)=O(h^{\alpha-1}e^{-{1}/{h^{\alpha}}})
\end{equation*}
as $n\rightarrow\infty.$

(ii) If $\phi_f(t)=O(|t|^{-\gamma})$ as $|t|\rightarrow\infty$ for $\gamma>1,$ then
\begin{equation*}
\ex[\hat{f}_{nh}(x)]-f(x)=O(h^{\gamma-1}).
\end{equation*}
as $n\rightarrow\infty.$
\end{prop}
\begin{rem}
\label{an-breakdown} Despite the fact that the bias\index{bias} of $\hat{f}_{nh}(x)$ asymptotically vanishes, the
consequence of Proposition \ref{deconvnoise-bias} is that the asymptotic normality\index{asymptotic normality} of
\eqref{deconvnoise-fnhf} cannot be established for the symmetric stable densities. Of course, it cannot be established
for other densities either, the characteristic functions of which decay algebraically. Examination of the proof of
Proposition \ref{deconvnoise-bias} demonstrates that in order to have that \eqref{deconvnoise-fnhf} is asymptotically
normal, one has to assume that, e.g., $\phi_f(t)=e^{-|t|^{\alpha}}$ with $\alpha>2.$ However if this is the case,
then $\phi_f^{\prime}(0)=\phi_f^{\prime\prime}(0)=0$ and consequently the first two moments of $f$ have to vanish.
There does not exist a density with such properties.
\end{rem}
\begin{rem}
It appears that in our case the square bias\index{bias!square} dominates the variance of the estimator. This is
not surprising in view of similar results obtained in \citet{butucea} for the ordinary
deconvolution\index{deconvolution} problem: suppose
\begin{equation*}
X=Y+Z,
\end{equation*}
where $Y$ and $Z$ are such that
\begin{eqnarray*}
\infint |\phi_Y(t)|^2\operatorname{exp}(2\alpha|t|^r)\leq 2\pi L,\\
b_{min}|t|^{\gamma}\operatorname{exp}(-\beta |t|^s)\leq \phi_Z(t)\leq b_{max}|t|^{\gamma^{'}}\operatorname{exp}(-\beta
|t|^s).
\end{eqnarray*}
Here $\alpha, r, L, b_{min}, r, b_{max},$ are strictly positive constants, $\gamma$ are $\gamma^{'}$ are real numbers,
and it is assumed that $r<s.$ Then the square bias\index{bias!square} of the deconvolution kernel density estimator,
which is based on observations on $X$ and is evaluted for the sinc kernel\index{kernel!sinc}, dominates the variance.
In our case $Y$ does not even have a characteristic function which vanishes at plus and minus infinity. The similarity
to the model in \citet{butucea} also holds true when comparing $\phi_f$ and $\phi_Z,$ as $\phi_Z,$ being the
characteristic function of a standard normal random variable, in a certain sense represents an extreme case among
characteristic functions of supersmooth distributions.
\end{rem}
\section{Simulation example}
\label{deconvnoise-simulation}

Practical implementation of the estimator \eqref{deconvnoise-hatfnh} is not a straightforward task. The idea we use is
similar to that of \citet{gug}. Notice that we can rewrite \eqref{deconvnoise-fnh} as
$f_{nh}(x)=f_{nh}^{(1)}(x)+f_{nh}^{(2)}(x),$ where
\begin{gather*}
f_{nh}^{(1)}(x)=\frac{1}{2\pi\lambda}\int_{0}^{\infty}
e^{-itx}\Log\left(\frac{\phi_{emp}(t)}{e^{-\lambda}e^{-t^2/2}}\right)dt,\\
f_{nh}^{(2)}(x)=\frac{1}{2\pi\lambda}\int_{0}^{\infty}
e^{itx}\Log\left(\frac{\phi_{emp}(-t)}{e^{-\lambda}e^{-t^2/2}}\right)dt.
\end{gather*}
Using the trapezoid rule and setting $v_j=\eta(j-1),$ $f_{nh}^{(1)}(x)$ can be approximated by
\begin{equation}
\label{deconvnoise-sum} f_{nh}^{(1)}(x)\approx\frac{1}{2\pi\lambda}\sum_{j=1}^N e^{-iv_j x}\psi(v_j)\eta.
\end{equation}
Here we take $N$ to be some power of $2$ and $\psi$ is defined by
\begin{equation*}
\psi(v_j)=\Log\left(\frac{\phi_{emp}(v_j)}{e^{-\lambda}e^{-t^2/2}}\right).
\end{equation*}

From this point on one can proceed as in \citet{gug} and evaluate \eqref{deconvnoise-sum} for a set of appropriately
selected points $x_1,\ldots x_N$ via the Fast Fourier Transform. A similar reasoning applies to $f_{nh}^{(2)}(x).$

The general difficulty with implementing the estimator is the computation of the distinguished
logarithm\index{logarithm!distinguished}, i.e.\ of function $\psi.$ A way to do this is to take a fine grid of points,
evaluate the argument of the ordinary logarithm there and if one sees large jumps of size comparable to $2\pi$ between
two consecutive points, make appropriate changes to the argument, thus obtaining an approximation to the argument of
the distinguished logarithm\index{logarithm!distinguished}. Of course this approach works only when $\phi_{emp}(t)$
does not vanish on $[-1/h,1/h].$ The latter fact can be verified in theory only, while in practice this can be done
only for a grid of points $t_1,t_2,\ldots t_k,$ which thus has to be taken rather fine, so that one does not possibly
miss the value zero.

Though our emphasis is more on theoretical aspects of decompounding under Gaussian noise, we nevertheless will consider
one simulation example in this section. We took $\lambda=1$ and $f$ the standard\index{density!standard normal} normal
density and simulated a sample of size $n=5000.$ The bandwidth\index{bandwidth} $h=0.5$ was selected by hand. The
resulting estimate $f_{nh}$ (bold dotted line) together with the true density $f$ (dashed line) is plotted in Figure
\vref{decompnoise_fig1}. We notice that the fit is quite good. Furthermore, notice that
\begin{equation*}
\operatorname{P}(N(\lambda)\geq 2)=1-2e^{-\lambda}\approx 0.264.
\end{equation*}
It turns out that we considered a nontrivial example, since a considerable number among the $Y_i$'s are {\it{sums}} of
the $W_j$'s in this case.

We should stress the fact that this simulation example serves as an illustration only and an extensive simulation study
is needed to investigate the finite sample performance of our estimator and its behaviour in practice. We have to be
very careful when generalising our conclusions concerning this simulation example because of the fact that the
empirical characteristic function\index{empirical characteristic function} is oscillatory in its tails. If the
integration step size $\eta$ is not small enough, we might miss instances when $\phi_{emp}$ crosses the negative real
axis. This will have direct consequences for the argument of the distinguished
logarithm\index{logarithm!distinguished}. This is especially true for relatively small sample sizes, for which the
empirical characteristic function $\phi_{emp}$\index{empirical characteristic function} might not approximate the true
characteristic function $\phi_X$ well enough. The issue of selection of $\eta$ in practice remains open and a thorough
simulation study is needed to obtain some practical recommendations how this can be done. Additionally, a
data-dependent method of the bandwidth\index{bandwidth} selection has to be created.

\begin{figure}[htb]
\begin{center}
\setlength{\unitlength}{1cm}
\begin{minipage}{6cm}
\begin{picture}(5.5,4.0)
\epsfxsize=5.5cm\epsfysize=4cm\epsfbox{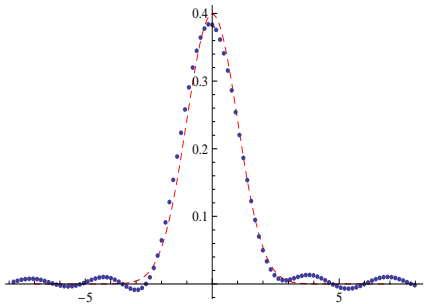}
\end{picture}
\par
\caption{\label{decompnoise_fig1} Estimation of the normal density, $n=5000.$}
\end{minipage}
\end{center}
\end{figure}

\section{Proofs}
\label{deconvnoise-proofs}
\begin{proof}[Proof of Corollary \ref{deconvnoise-log}]
Note that we have
\begin{equation*}
\operatorname{P}(B_{nh})\leq 4
\left(1+\frac{8\ex[X]}{e^{-\lambda}\delta}\frac{e^{1/(2h^2)}}{h}\right)\operatorname{exp}\left(\frac{-e^{-2\lambda}\delta^2ne^{-1/h^2}}{72}\right)+o(1),
\end{equation*}
where we assume that $\delta$ is small enough. This bound follows from Theorem~\ref{devroye} with $\alpha=1/h$ and
$\beta=e^{-\lambda}\delta e^{-1/(2h^2)}.$ The right-hand side converges to zero as $n\rightarrow\infty$ due to
Condition \ref{conditionh}. To prove the corollary, the only additional fact that we need to verify is that the $o(1)$
term from Theorem \ref{devroye}, which in the proof of Theorem \ref{devroye} from \citet{devroye-chiu} is bounded by
\eqref{o1term},
is of order $n^{1-\rho},$ if $1<\rho<2,$ and is of order $n^{-\rho/2},$ if $\rho\geq2.$ In fact, if the inequality
\begin{equation*}
\left|\frac{1}{n}\sum_{j=1}^nX_j\right|\geq\frac{4}{3}\ex[|X_1|]
\end{equation*}
holds, then we have
\begin{multline*}
\left|\frac{1}{n}\sum_{j=1}^nX_j-\ex[X_1]\right|\geq\left|\frac{1}{n}\sum_{j=1}^nX_j\right|-\left|\ex[X_1]\right|\\
\geq\frac{4}{3}\ex[|X_1|]-\ex[|X_1|]=\frac{1}{3}\ex[|X_1|].
\end{multline*}
By Chebyshev's inequality\index{inequality!Chebyshev's} this implies
\begin{multline}
\label{chebyshev}
\operatorname{P}\left(\left|\frac{1}{n}\sum_{j=1}^nX_j\right|\geq\frac{4}{3}\ex[|X_1|]\right)\leq\operatorname{P}\left(\left|\frac{1}{n}\sum_{j=1}^nX_j-\ex[X_1]\right|\geq\frac{1}{3}\ex[|X_1|]\right)\\
\leq 3^{\rho}(\ex[|X_1|])^{-\rho}\ex\left[\left|\frac{1}{n}\sum_{j=1}^nX_j-\ex[X_1]\right|^{\rho}\right].
\end{multline}
Suppose first that $1<\rho<2.$ Then it follows from Theorem 4 of \citet{bahr} that the rightmost term in
\eqref{chebyshev} is of order $n^{1-\rho}.$ Now suppose $\rho\geq 2.$ Then Theorem 2 of \citet{dharmadhikari} implies
that the rightmost term of \eqref{chebyshev} is of order $n^{-\rho/2}.$ For explicit constants we refer to the same
papers.

Assume again that $1<\rho<2.$ To complete the proof of the corollary, we have to verify that
\begin{equation}
\label{rem1} \frac{e^{{1}/({2h^2})}}{h}\operatorname{exp}\left(\frac{-e^{-2\lambda}\delta^2ne^{-1/h^2}}{72}\right)n^{\rho-1}\rightarrow 0.
\end{equation}
To this end we take the logarithm of the left-hand side to obtain
\begin{equation}
\label{rem2} \frac{1}{2h^2}-\log h-\frac{e^{-2\lambda}\delta^2}{72}ne^{-{1}/{h^2}}+(\rho-1)\log n.
\end{equation}
The first term here is of order $(\log n)^{2\beta}$ and is negligible compared to $\log n.$ The second term is of order
$\log\log n$ and is thus negligible, while the third term dominates $\log n.$ Therefore \eqref{rem2} diverges to minus
infinity and consequently \eqref{rem1} holds. The proof for the case $\rho\geq 2$ is virtually identical and therefore
it is omitted.
\end{proof}
The following lemma will be used in the proof of Theorem \ref{deconvnoise-an1}.
\begin{lem}
\label{Bertlemma} Assume the conditions of Theorem \ref{deconvnoise-an1}. Let
\begin{equation*}
{f}_{nh}^{*}(x)=\frac{1}{2\pi}\int_{-1/h}^{1/h}e^{-itx}\frac{\phi_{emp}(t)}{e^{-t^2/2}}dt.
\end{equation*}
Then
\begin{equation*}
\frac{\sqrt{n}}{he^{1/(2h^2)}}{f}_{nh}^{*}(x)-\ex[{f}_{nh}^{*}(x)]\rightarrow {\mathcal
N}\left(0,\frac{e^{2\lambda}}{2\pi^2\lambda^2}\right)
\end{equation*}
as $n\rightarrow 0.$
\end{lem}
\begin{proof}
The proof is a minor variation of the proof of Theorem $2.1$ of \citet{vanes2}. The arguments of \citet{vanes2} are
applicable, because they only use the existence and continuity of the density $q$ of $X,$ which is still true in our
case.

Let $S$ denote a random variable,
independent of the $X$'s and with a density
\begin{equation*}
f_S(x)=\frac{1}{c(h)}e^{x^2/(2h^2)}1_{[0,1]},
\end{equation*}
where the normalisation constant $c(h)=\int_0^1 e^{x^2/(2h^2)}dx.$ Furthermore, let $E_n=(S-1)/h^2$ and let $E$ be a
standard exponential random variable. If for a random variable $S$ we set
\begin{equation*}
\ex\left(\cos\left(\frac{S}{h}(X_j-x)\right)|X_j\right)=\ex_S\left[\cos\left(\frac{S}{h}(X_j-x)\right)\right],
\end{equation*}
then ${f}_{nh}^{*}(x)-\ex[{f}_{nh}^{*}(x)]$ can be written as
\begin{align}
f_{nh}^{*}(x)-\ex [f_{nh}^{*}(x)] &= \frac{c(h)}{\pi nh}\sum_{j=1}^n \Big( \cos\Big(\frac{X_j-x}{h}\Big)\ex_{E}\left[\cos\Big(-h
E(X_j-x)\Big)\right]\nonumber\\
& -\sin\Big(\frac{X_j-x}{h}\Big)\ex_{E}\left[\sin\Big(-h E(X_j-x)\Big)\right]\Big)\nonumber\\
& -\ex\left[\cos\Big(\frac{X_j-x}{h}\Big)\ex_{E}\left[\cos\Big(-hE(X_j-x)\Big)\right]\right]\nonumber\\
& +\ex\left[\sin\Big(\frac{X_j-x}{h}\Big)\ex_{E}\left[\sin\Big(-h
E(X_j-x)\Big)\right]\right]\Big)\nonumber
\\
& + O_P\left(\frac{1}{\sqrt{n}}h^{3}e^{1/(2h^2)}\right), \label{approxerror}
\end{align}
cf.\ equation (31) of \citet{vanes2}.

A straightforward computation yields
\begin{equation*}
\ex_E\left[\cos\Big(-hE(X_j-x)\Big)\right] = \frac{1}{{1+h^{2}(X_j-x)^2}} =w_1(h(X_j-x))
\end{equation*}
and
\begin{equation*}
\ex_E\left[\sin\Big(-h^E(X_j-x)\Big)\right] = -\frac{h(X_j-x)}{1+h^{2}(X_j-x)^2} =w_2(h(X_j-x)),
\end{equation*}
where
\begin{equation*}
w_1(u)=\frac{1}{1+u^2}\quad\mbox{and}\quad w_2(u)=- \frac{u}{1+u^2}.
\end{equation*}

Define the random variables $V_{nj}$   as
\begin{align}
V_{nj}&=\cos\Big(\frac{X_j-x}{h}\Big)w_1(h(X_j-x))
-\sin\Big(\frac{X_j-x}{h}\Big)w_2(h(X_j-x))\nonumber\\
&= \cos(Y_{h,j})w_1(h(X_j-x)) -\sin(Y_{h,j})w_2(h(X_j-x)),
\end{align}
where $Y_{h,j}= (X_j-x)/h\operatorname{mod}2\pi$. Then
\begin{equation}\label{fasv}
{f}_{nh}^{*}(x)-\ex [{f}_{nh}^{*}(x)]=\frac{c(h)}{\pi h}\,\frac{1}{n}\sum_{j=1}^n (V_{n,j}-\ex [V_{n,j}])+
O_P\left(\frac{1}{\sqrt{n}}h^3e^{1/(2h^2)}\right).
\end{equation}
Note that by the inequality $|a+b|^p\leq 2^p(|a|^p+|b|^p), p\geq 0$, we have
\begin{equation}
\label{lyap1} \ex [(V_{n,j}-\ex [V_{n,j}])^4]\leq 16(\ex [V_{n,j}^4]+(\ex [V_{n,j}])^4).
\end{equation}

Since the characteristic function $\phi_X$ is integrable, by \citet[Theorem~$6.2.3$]{chung} the density $q$ of $X$ is
continuous and bounded. Hence $X$ satisfies the conditions of Lemma $3.1$ of \citet{vanes2}, and we have
$(hX,Y_{h,j})\convd (0,U)$. It also holds that
\begin{align*}
\ex [V_{nj}]&=\ex [\cos(Y_{h,j})w_1(h(X_j-x)) -\sin(Y_{h,j})w_2(h(X_j-x))]\nonumber\\ &\rightarrow
\ex [\cos(U)w_1(0) -\sin(U)w_2(0)]=0
\end{align*}
and
\begin{align*}
\ex [V_{nj}^4]&=\ex [(\cos(Y_{h,j})w_1(h(X_j-x)) -\sin(Y_{h,j})w_2(h(X_j-x)))^4]\nonumber\\ &\rightarrow \ex[
(\cos(U)w_1(0) -\sin(U)w_2(0))^4]=\ex [(\cos(U))^4]=\frac{3}{8},
\end{align*}
because the cosine is a bounded and continuous function. The asymptotic variance of $V_{n,j}$ is given by
\begin{equation*}
\var [V_{n,j}]\rightarrow \ex [(\cos(U)w_1(0) -\sin(U)w_2(0))^2]=\ex [(\cos(U))^2]=\frac{1}{2}.
\end{equation*}
It follows from \eqref{lyap1} that
\begin{equation}
\frac{\ex [|V_{n,j}-\ex [V_{n,j}]|^4]}{n(\var[V_{n,j}])^2} =\frac{O(1)}{n(\frac{1}{2}+o(1))^2}\rightarrow 0.
\end{equation}
Consequently, Lyapunov's condition with $\delta=2$ is satisfied for $V_{n,j},$ and hence both ${1}/{n}\sum_{j=1}^n
(V_{n,j}-\ex [V_{n,j}])$ and ${c(h)}/{(\pi h n)}\sum_{j=1}^n (V_{n,j}-\ex [V_{n,j}])$ are asymptotically normal. The
asymptotic variance of the latter is given by
\begin{equation*}
\var\left[\frac{c(h)}{\pi h}\frac{1}{n}\sum_{j=1}^n (V_{n,j}-\ex [V_{n,j}])\right]= \frac{1}{n}\frac{c(h)^2}{\pi^2
h^2}\,\var[V_{n,1}] \sim \frac{1}{n}\frac{1}{2\pi^2}h^{2}\,e^{1/h^2},
\end{equation*}
which follows from Lemma $2.1$ of \citet{vanes2}. This completes the proof of the lemma.
\end{proof}
\begin{proof}[Proof of Theorem \ref{deconvnoise-an1}]
Write $\zeta_n(h)=\sqrt{n}h^{-1}e^{-1/(2h^2)}.$ We have
\begin{multline*}
\zeta_n(h)(\hat{f}_{nh}(x)-\operatorname{E}[\hat{f}_{nh}(x)])=\zeta_n(h)(\hat{f}_{nh}(x)-f(x))+\zeta_n(h)(f(x)-\operatorname{E}[\hat{f}_{nh}(x)])\\
=\zeta_n(h)((\hat{f}_{nh}(x)-f(x))1_{B_{nh}}-\operatorname{E}[(\hat{f}_{nh}(x)-f(x))1_{B_{nh}}])\\
+\zeta_n(h)((\hat{f}_{nh}(x)-f(x))1_{B_{nh}^c}-\operatorname{E}[(\hat{f}_{nh}(x)-f(x))1_{B_{nh}^c}]),
\end{multline*}
where the set $B_{nh}$ is defined as in Section \ref{deconvnoise-results}. Now notice that, for an arbitrary constant
$\eta>0,$ by Chebyshev's inequality\index{inequality!Chebyshev's} we have
\begin{multline}
\label{deconvnoise-asn1}
\operatorname{P}(\zeta_n(h)|(\hat{f}_{nh}(x)-f(x))1_{B_{nh}}-\operatorname{E}[(\hat{f}_{nh}(x)-f(x))1_{B_{nh}}]|>\eta)\\
\leq\frac{2}{\eta}\zeta_n(h)\operatorname{E}[|\hat{f}_{nh}(x)-f(x)|1_{B_{nh}}].
\end{multline}
Since $\phi_f$ is integrable, it follows that $|f(x)|\leq C,$ where $C$ is some constant. It then follows that the
probability at the left-hand side of \eqref{deconvnoise-asn1} is bounded by
\begin{equation}
\label{deconvnoise-asn2} \frac{2}{\eta}\zeta_n(h)(M_n+C)\operatorname{P}(B_{nh}).
\end{equation}
Now we apply the bound of Proposition \ref{deconvnoise-log} to $\operatorname{P}(B_{nh}).$ To prove that
\eqref{deconvnoise-asn2} converges to zero, it is sufficient to verify that
\begin{equation}
\label{log0} \frac{\sqrt{n}}{he^{{1}/{(2h^2)}}}(M_n+C)\frac{1}{n^{\rho-1}}\rightarrow 0
\end{equation}
for $3/2<\rho<2.$ This is obviously true due to Condition \ref{conditionh} and \ref{conditionm}. The proof that
\eqref{deconvnoise-asn1} converges to zero for $\rho\geq 2$ is likewise straightforward. Therefore
\begin{equation*}
\zeta_n(h)\{(\hat{f}_{nh}(x)-f(x))1_{B_{nh}}-\operatorname{E}[(\hat{f}_{nh}(x)-f(x))1_{B_{nh}}]\}\convp 0.
\end{equation*}
Hence by Slutsky's\index{theorem!Slutsky's} theorem, see \citet[Lemma~2.8]{vaart}, this term can be neglected and it
suffices to consider
\begin{equation*}
\zeta_n(h)\{(\hat{f}_{nh}(x)-f(x))1_{B_{nh}^c}-\operatorname{E}[(\hat{f}_{nh}(x)-f(x))1_{B_{nh}^c}]\}.
\end{equation*}

We have that
\begin{equation*}
\log\left(\left|\frac{\phi_X(t)}{e^{-\lambda}e^{-t^2/2}}\right|\right),
\end{equation*}
i.e. the real part of the distinguished logarithm
\begin{equation}
\label{dlog1} \Log\left(\frac{\phi_X(t)}{e^{-\lambda}e^{-t^2/2}}\right)
\end{equation}
is bounded.  On the set $B_{nh}^c,$ if $\delta$ is selected small enough, $\phi_{emp}(t)e^{\lambda}e^{t^2/2}$ is
arbitrarily close to $\phi_{X}(t)e^{\lambda}e^{t^2/2}$ and also stays bounded away from zero at a positive distance.
Therefore
\begin{equation*}
\int_{-1/h}^{1/h}\log\left(\left|\frac{\phi_{emp}(t)}{e^{-\lambda}e^{-t^2}}\right|\right)dt\leq C\frac{1}{h},
\end{equation*}
where $C$ is a constant. This grows slower than $M_n$ and hence $M_n$ will eventually dominate. Now we turn to the
imaginary part. Let $\psi:{\mathbb{R}}\rightarrow {\mathbb{C}},$ where
\begin{equation*}
\psi(t)=\phi_X(t)e^{\lambda}e^{t^2/2}=e^{\lambda\phi_f(t)}.
\end{equation*}
By the Riemann-Lebesgue theorem $\psi(t)$ converges to $1$ as $|t|\rightarrow\infty$ and hence there exists $t^{*}>0,$
such that
\begin{equation}
\label{cut1} |\psi(t)-1|<\frac{1}{2}, \quad |t|>t^{*}.
\end{equation}
Furthermore, we have
\begin{equation}
\label{cut2} |\psi(t)|\geq e^{-\lambda}, \quad t\in {\mathbb{R}}.
\end{equation}
Since $f$ has a finite first moment, by \citet[Theorem~1,~p.~182]{schwartz} $\phi_f$ and consequently $\psi$ are
continuously differentiable. Therefore, the path $\psi:[-t^{*},t^{*}]\rightarrow {\mathbb{C}}$ is rectifiable, i.e.\
has a finite length. In view of this fact and \eqref{cut2}, $\psi:[-t^{*},t^{*}]\rightarrow {\mathbb{C}}$ cannot spiral
infinitely many times around zero and for $|t|>t^{*}$ it cannot make a turn around zero at all because of \eqref{cut1}.
Consequently, for the same reason as we gave above for the real part of the distinguished logarithm, the truncation on
this set becomes unimportant for the argument as well and we have
\begin{equation}
\label{hateq} \hat{f}_{nh}(x)1_{B_{nh}^c}=f_{nh}(x)1_{B_{nh}^c}.
\end{equation}

Thus we have to consider
\begin{equation*}
\zeta_n(h)\{(f_{nh}(x)-f(x))1_{B_{nh}^c}-\operatorname{E}[(f_{nh}(x)-f(x))1_{B_{nh}^c}]\}.
\end{equation*}
Plugging in the expressions for $f_{nh}(x)$ and $f(x),$ we obtain that the above expression is equal to
\begin{multline}
\label{c} \zeta_n(h) \Biggl\{
\frac{1}{2\pi\lambda}\int_{-1/h}^{1/h}e^{-itx}\Log\left(\frac{\phi_{emp}(t)}{\phi_X(t)}\right)dt1_{B_{nh}^c}\\
-\ex\left[\frac{1}{2\pi\lambda}\int_{-1/h}^{1/h}e^{-itx}\Log\left(\frac{\phi_{emp}(t)}{\phi_X(t)}\right)dt1_{B_{nh}^c}\right]\\
-\frac{1}{2\pi}\Biggl\{\int_{-\infty}^{-1/h}e^{-itx}\phi_f(t)dt+\int_{1/h}^{\infty}e^{-itx}\phi_f(t)dt \Biggr\}
(1_{B_{nh}^c}-\ex[1_{B_{nh}^c}]).
\end{multline}
First notice that
\begin{equation*}
\left|\int_{-\infty}^{-1/h}e^{-itx}\phi_f(t)dt+\int_{1/h}^{\infty}e^{-itx}\phi_f(t)dt\right|\leq\int_{\infty}^{\infty}|\phi_f(t)|dt<\infty.
\end{equation*}
Consequently, the last term in \eqref{c} converges to zero in probability if
\begin{equation*}
\zeta_n(h)(1_{B_{nh}^c}-\ex[1_{B_{nh}^c}])\convp 0.
\end{equation*}
This in turn is equivalent to
\begin{equation*}
\zeta_n(h)(1_{B_{nh}}-\ex[1_{B_{nh}}])\convp 0,
\end{equation*}
because $1_{B_{nh}^c}=1-1_{B_{nh}}.$ By Chebyshev's inequality\index{inequality!Chebyshev's} it is sufficient to prove
that $\zeta_n(h)\operatorname{P}(B_{nh})\rightarrow 0.$ However, this follows from \eqref{deconvnoise-asn2} and
\eqref{log0}.

Hence, by Slutsky's\index{theorem!Slutsky's} theorem we have to consider the first term of \eqref{c},
\begin{multline*}
\zeta_n(h) \Biggl\{
\frac{1}{2\pi\lambda}\int_{-1/h}^{1/h}e^{-itx}\Log\left(\frac{\phi_{emp}(t)}{\phi_X(t)}\right)dt1_{B_{nh}^c}\\
-\ex\left[\frac{1}{2\pi\lambda}\int_{-1/h}^{1/h}e^{-itx}\Log\left(\frac{\phi_{emp}(t)}{\phi_X(t)}\right)dt1_{B_{nh}^c}\right]
\Biggr\}.
\end{multline*}
Rewrite this as
\begin{multline}
\label{deconvnoise-asnorm1}
\zeta_{n}(h) \Biggl\{
\frac{1}{2\pi\lambda}\int_{-1/h}^{1/h} e^{-itx}\left(\frac{\phi_{emp}(t)}{\phi_X(t)}-1\right)dt1_{B_{nh}^c}\\
-\ex\left[\frac{1}{2\pi\lambda}\int_{-1/h}^{1/h}e^{-itx}\left(\frac{\phi_{emp}(t)}{\phi_X(t)}-1\right)dt1_{B_{nh}^c}\right]\Biggr\}\\
+\zeta_{n}(h)\Biggl\{\frac{1}{2\pi\lambda}\int_{-1/h}^{1/h}e^{-itx}R_{nh}(t)dt1_{B_{nh}^c}\\
-\ex\left[\frac{1}{2\pi\lambda}\int_{-1/h}^{1/h}e^{-itx}R_{nh}(t)dt1_{B_{nh}^c}\right] \Biggr\},
\end{multline}
where
\begin{equation}
\label{deconvnoise-rnh}
R_{nh}(t)=\operatorname{Log}\left(1+\left\{\frac{\phi_{emp}(t)}{\phi_X(t)}-1\right\}\right)-\left\{\frac{\phi_{emp}(t)}{\phi_X(t)}-1\right\}.
\end{equation}
Notice that on the set $B_{nh}^c$ we have
\begin{equation*}
\left|\frac{\phi_{emp}(t)}{\phi_X(t)}-1\right|<\frac{1}{2},
\end{equation*}
if $\delta$ is small enough. Indeed, it suffices to choose $\delta$ in such a way that $e^{\lambda}\delta<1/2.$ From
the inequality $|\Log(1+z)-z|\leq |z|^2,$ valid for $z$ sufficiently small, it
follows 
that
\begin{equation*}
|R_{nh}(t)|\leq \left|\frac{\phi_{emp}(t)}{\phi_X(t)}-1\right|^2.
\end{equation*}
Consequently, to prove that the second term in \eqref{deconvnoise-asnorm1} asymptotically vanishes, it is sufficient to
prove that
\begin{multline}
\label{rnhbound}
\zeta_{n}(h)\frac{1}{2\pi\lambda}\ex\left[\int_{-1/h}^{1/h}|R_{nh}(t)|dt1_{B_{nh}^c}\right]\\
\leq\zeta_{n}(h)\frac{1}{2\pi\lambda}
\ex\left[\int_{-1/h}^{1/h}\left|\frac{\phi_{emp}(t)}{\phi_X(t)}-1\right|^2dt\right]\rightarrow 0.
\end{multline}
Since ${|\phi_Y(t)|^{-1}}\leq e^{2\lambda}$, we have
\begin{multline*}
\ex\left[\int_{-1/h}^{1/h}\left|\frac{\phi_{emp}(t)}{\phi_X(t)}-1\right|^2dt\right]\\
\leq C
e^{{1}/{h^2}}\ex\left[\infint|\phi_{emp}(t)\phi_w(ht)-\phi_X(t)\phi_w(ht)|^2dt\right],
\end{multline*}
where $\phi_w$ is the characteristic function of the sinc kernel\index{kernel!sinc} and $C$ is a constant. By
Parseval's identity the expectation on the right-hand side equals
\begin{equation*}
\frac{1}{2\pi}\ex\left[\infint(q_{nh}(x)-q\ast w_h(x))^2dx\right].
\end{equation*}
This in turn equals the integrated variance of a kernel estimator $q_{nh},$ which is of order $(nh)^{-1},$ see
\citet[Proposition~1.7]{tsyb}. Thus we have to show that
\begin{equation*}
\frac{\sqrt{n}}{he^{{1}/{(2h^2)}}}e^{{1}/{h^2}}\frac{1}{nh}=\frac{e^{{1}{(2h^2)}}}{h^2\sqrt{n}}\rightarrow 0.
\end{equation*}
The result follows from Condition \ref{conditionh} and can be verified by taking the logarithm of the left-hand side of
the above expression and concluding that it diverges to minus infinity. We obtain
\begin{equation*}
\frac{1}{2h^2}-\log h^2-\frac{1}{2}\log n\rightarrow -\infty,
\end{equation*}
because $h^{-2}=(\log n)^{2\beta}$ and $2\beta<1,$ and hence the dominating term on the left-hand side in the above
expression is the last one.

We deal with the first summand in \eqref{deconvnoise-asnorm1}. Rewrite it as
\begin{multline}
\label{deconvnoise-asnorm2} \zeta_{n}(h)
\Biggl\{\frac{1}{2\pi\lambda}\int_{-1/h}^{1/h}e^{-itx}\frac{\phi_{emp}(t)}{\phi_X(t)}dt1_{B_{nh}^c}\\
-\ex\left[\frac{1}{2\pi\lambda}\int_{-1/h}^{1/h}e^{-itx}\frac{\phi_{emp}(t)}{\phi_X(t)}dt1_{B_{nh}^c}\right]\Biggr\}\\
-\zeta_{n}(h)\frac{1}{2\pi\lambda}\int_{-1/h}^{1/h}e^{-itx}dt(1_{B_{nh}^c}-\ex[1_{B_{nh}^c}]).
\end{multline}
We want to show that the second summand in this expression converges to zero in probability. Notice, that it is bounded
by
\begin{equation*}
C\zeta_{n}(h)\frac{1}{h}|1_{B_{nh}}-\ex[1_{B_{nh}}]|,
\end{equation*}
because $1_{B_{nh}^c}=1-1_{B_{nh}}.$ Here $C$ is some constant. By Chebyshev's inequality\index{inequality!Chebyshev's}
it is sufficient to prove that
\begin{equation*}
\zeta_{n}(h)\frac{1}{h}\operatorname{P}(B_{nh})\rightarrow 0.
\end{equation*}
This is obviously true thanks to \eqref{deconvnoise-asn1} and \eqref{deconvnoise-asn2}.

Thus, by Slutsky's\index{theorem!Slutsky's} theorem, instead of \eqref{deconvnoise-asnorm2} we may consider
\begin{equation*}
\zeta_{n}(h)\Biggl\{\frac{1}{2\pi\lambda}\int_{-1/h}^{1/h}e^{-itx}\frac{\phi_{emp}(t)}{\phi_X(t)}dt1_{B_{nh}^c}
-\ex\left[\frac{1}{2\pi\lambda}\int_{-1/h}^{1/h}e^{-itx}\frac{\phi_{emp}(t)}{\phi_X(t)}dt1_{B_{nh}^c}\right]\Biggr\}.
\end{equation*}
Note that for $|t|\leq 1/h$
\begin{equation*}
|\phi_X(t)|=\left|e^{-{t^2}/{2}}e^{-\lambda+\lambda\phi_f(t)}\right|\geq e^{-{1}/{(2h^2)}}e^{-2\lambda}
\end{equation*}
holds. Consequently, we have
\begin{multline}
\label{subs1}
\zeta_{n}(h)\ex\left[\left|\frac{1}{2\pi\lambda}\int_{-1/h}^{1/h}e^{-itx}\frac{\phi_{emp}(t)}{\phi_X(t)}dt1_{B_{nh}}\right|\right]\\
\leq\zeta_{n}(h)\frac{1}{2\pi\lambda}\frac{2}{h}e^{{1}/{(2h^2)}}e^{2\lambda}\operatorname{P}(B_{nh}),
\end{multline}
which converges to zero thanks to the fact that $\operatorname{P}(B_{nh})=O(n^{1-\rho})$ for $3/2<\rho<2,$ and
$\operatorname{P}(B_{nh})=O(n^{-\rho/2})$ for $\rho\geq 2,$ see Proposition \ref{deconvnoise-log}.

Hence by Slutsky's\index{theorem!Slutsky's} theorem we may consider
\begin{equation*}
\zeta_{n}(h)\Biggl\{\frac{1}{2\pi\lambda}\int_{-1/h}^{1/h}e^{-itx}\frac{\phi_{emp}(t)}{\phi_X(t)}dt
-\ex\left[\frac{1}{2\pi\lambda}\int_{-1/h}^{1/h}e^{-itx}\frac{\phi_{emp}(t)}{\phi_X(t)}dt\right]\Biggr\}.
\end{equation*}
By \eqref{deconv-phix} the expression above can be rewritten as
\begin{multline}
\label{deconvnoise-asnorm3}
\zeta_{n}(h)\frac{e^{\lambda}}{\lambda}\frac{1}{2\pi}\int_{-1/h}^{1/h}e^{-itx}\left(\frac{\phi_{emp}(t)}{e^{-t^2/2}}
-\frac{\phi_X(t)}{e^{-t^2/2}}\right)dt\\+
\zeta_{n}(h)\frac{e^{\lambda}}{\lambda}\frac{1}{2\pi}\int_{-1/h}^{1/h}e^{-itx}\left(\frac{\phi_{emp}(t)}{e^{-t^2/2}}-
\frac{\phi_X(t)}{e^{-t^2/2}}\right)\left(e^{-\lambda\phi_f(t)}-1\right)dt.
\end{multline}
By Lemma \ref{Bertlemma} the first summand in this expression is asymptotically normal with zero mean and variance
given by $\sigma^2={e^{2\lambda}}/{(2\pi^2\lambda^2)}.$

Now we will show that the second term in \eqref{deconvnoise-asnorm3} asymptotically vanishes in probability. By
Chebyshev's inequality\index{inequality!Chebyshev's} it suffices to show
\begin{multline*}
(\zeta_{n}(h))^2\ex\left[\left|\frac{e^{\lambda}}{\lambda}\frac{1}{2\pi}\int_{-1/h}^{1/h}e^{-itx}\left(\frac{\phi_{emp}(t)}{e^{-t^2/2}}-
\frac{\phi_X(t)}{e^{-t^2/2}}\right)\left(e^{-\lambda\phi_f(t)}-1\right)dt\right|^2\right]\\
= (\zeta_{n}(h))^2\operatorname{Var}\left[\int_{-1/h}^{1/h}e^{-itx}\frac{\phi_{emp}(t)}{e^{-t^2/2}}
\left(e^{-\lambda\phi_f(t)}-1\right)dt\right]\rightarrow 0.
\end{multline*}
Using the independence of the $X_i$'s, after further simplification we obtain that we have to prove that
\begin{equation*}
\frac{1}{h^2e^{1/h^2}}\left(\int_{-1/h}^{1/h}e^{t^2/2}|e^{-\lambda\phi_f(t)}-1|dt\right)^2\rightarrow 0.
\end{equation*}
Thus we have to prove that
\begin{equation*}
\frac{1}{he^{(1/2h^2)}}\int_{-1/h}^{1/h}e^{t^2/2}|e^{-\lambda\phi_f(t)}-1|dt\rightarrow 0.
\end{equation*}
From \citet{gug} we have
\begin{equation*}
|e^{-\lambda\phi_f(t)}-1|\leq C_{\lambda}|\phi_f(t)|,
\end{equation*}
where the constant $C_{\lambda}$ depends on $\lambda$ only. Therefore it suffices to prove
\begin{equation}
\label{tail1} \frac{1}{he^{1/(2h^2)}}\int_{-1/h}^{1/h}e^{t^2/2}|\phi_f(t)|dt\rightarrow 0.
\end{equation}
This can be done either via an application of L'H\^opital's rule\index{rule!L'H\^opital's} or via the method similar to
the one used in the proof of Lemma 5 in \citet{vanes3}. We follow the latter path. It is enough to consider the
integral over $[0,1/h]$ as the integral over $[-1/h,0]$ can be dealt with in a similar fashion. After the change of
integration variable $v=(1-ht)/h^2,$ we obtain
\begin{equation*}
h\int_0^{1/h^2}e^{\frac{(1-h^2v)^2}{2h^2}}\phi_f\left(\frac{1-vh^2}{h}\right)dv=
he^{\frac{1}{2h^2}}\int_0^{1/h^2}e^{-v+\frac{v^2h^2}{2}}\phi_f\left(\frac{1-vh^2}{h}\right)dv.
\end{equation*}
By the Riemann-Lebesgue theorem\index{theorem!Riemann-Lebesgue's} $\lim_{|u|\rightarrow\infty}\phi_f(u)=0,$ and
therefore by the dominated convergence theorem\index{theorem!dominated convergence} the above expression is of order
$o(he^{{1}/{(2h^2)}}).$ The dominated convergence theorem\index{theorem!dominated convergence} is applicable, because
\begin{equation*}
(e^{-v/2}-e^{-v+v^2h^2/2})1_{[0,1/h^2]}\geq 0
\end{equation*}
and hence $e^{-v/2}$ can be taken as the dominating function. Consequently \eqref{tail1} vanishes as $h\rightarrow 0$
and this argument concludes the proof of the theorem.

\end{proof}
\begin{proof}[Proof of Proposition \ref{deconvnoise-bias}]
We will prove both parts of the statement simultaneously. Write
\begin{equation}
\label{deconvnoise-b0}
\ex[\hat{f}_{nh}(x)]-f(x)=\ex[(\hat{f}_{nh}(x)-f(x))1_{B_{nh}}]+\ex[(\hat{f}_{nh}(x)-f(x))1_{B_{nh}^c}].
\end{equation}
Notice that for some $C>0,$
\begin{equation}
\label{deconvnoise-b1} \left|\ex[(\hat{f}_{nh}(x)-f(x))1_{B_{nh}}]\right|\leq (M_n+C)\operatorname{P}(B_{nh}).
\end{equation}
Here we used the fact that $f$ is bounded, because $\phi_f$ is integrable. Due to Theorem \ref{devroye}, Corollary
\ref{deconvnoise-log} and Conditions \ref{conditionh} and \ref{conditionm}, we see that \eqref{deconvnoise-b1}
converges to zero as $n\rightarrow\infty.$ Moreover, this term is negligible compared to
$h^{\alpha-1}e^{-1/h^{\alpha}}$ (case (i)) or $h^{\gamma-1}$ (case (ii)), since
$\operatorname{P}(B_{nh})=O(n^{1-\rho})$ or $\operatorname{P}(B_{nh})=O(n^{-\rho/2}),$ depending whether $1<\rho<2$ or
$\rho\geq 2,$ see Proposition \ref{deconvnoise-log}.

Now we turn to the second summand in \eqref{deconvnoise-b0}. By selecting $\delta$ small enough, on the set $B_{nh}^c$
truncation in the definition of $\hat{f}_{nh}(x)$ becomes unimportant, see the arguments that led to \eqref{hateq}.
Hence we have to deal with $\ex[({f}_{nh}(x)-f(x))1_{B_{nh}^c}].$ Using expressions for $f_{nh}(x)$ and $f(x),$ we see
that this term equals
\begin{multline}
\label{deconvnoise-b3}
\ex\left[\frac{1}{2\pi\lambda}\int_{-1/h}^{1/h}e^{-itx}\Log\left(\frac{\phi_{emp}(t)}{\phi_X(t)}\right)dt1_{B_{nh}^c}\right]\\-\frac{1}{2\pi}\int_{-\infty}^{-1/h}e^{-itx}\phi_f(t)dtP(B_{nh}^c)-\frac{1}{2\pi}\int_{1/h}^{\infty}e^{-itx}\phi_f(t)dt\operatorname{P}(B_{nh}^c).
\end{multline}
The last two terms in this expression can be treated similarly and therefore we consider only the second one. It will
turn out that these are the leading terms in the bias\index{bias!expansion} expansion. Notice that
\begin{equation*}
\frac{1}{2\pi}\int_{1/h}^{\infty}e^{-itx}\phi_f(t)dt\operatorname{P}(B_{nh}^c)\rightarrow 0
\end{equation*}
as $n\rightarrow 0,$ because $\phi_f$ is integrable. Moreover, if $\phi_f(t)=O\left(e^{-|t|^{\alpha}}\right),$
$\alpha>1,$ then
\begin{equation*}
\label{breakdown} \int_{1/h}^{\infty}|\phi_f(t)|dt= O\left({h^{\alpha-1}e^{-{1}/{h^{\alpha}}}}\right).
\end{equation*}
This fact can be proved using the same type of arguments as in \citet[Example~3.6.3,~p.~123]{casella}. Furthermore, if
$\phi_f(t)= O\left(|t|^{-\gamma}\right),$ then
\begin{equation*}
\int_{1/h}^{\infty}|\phi_f(t)|dt\leq C \int_{1/h}^{\infty}t^{-\gamma}dt=O(h^{\gamma-1}).
\end{equation*}

Now we turn to the first term in \eqref{deconvnoise-b3}. Rewrite it as
\begin{multline}
\label{deconvnoise-b4}
\ex\left[\frac{1}{2\pi\lambda}\int_{-1/h}^{1/h}e^{-itx}\left(\frac{\phi_{emp}(t)}{\phi_X(t)}-1\right)dt1_{B_{nh}^c}\right]\\
+\ex\left[\frac{1}{2\pi\lambda}\int_{-1/h}^{1/h}e^{-itx}R_{nh}(t)dt1_{B_{nh}^c}\right],
\end{multline}
where $R_{nh}(t)$ is defined as in \eqref{deconvnoise-rnh}. Consider the first term in this expression. Rewrite it as
\begin{multline*}
\ex\left[\frac{1}{2\pi\lambda}\int_{-1/h}^{1/h}e^{-itx}\left(\frac{\phi_{emp}(t)}{\phi_X(t)}-1\right)dt1_{B_{nh}^c}\right]\\
=\ex\left[\frac{1}{2\pi\lambda}\int_{-1/h}^{1/h}e^{-itx}\left(\frac{\phi_{emp}(t)}{\phi_X(t)}-1\right)dt\right]\\-
\ex\left[\frac{1}{2\pi\lambda}\int_{-1/h}^{1/h}e^{-itx}\left(\frac{\phi_{emp}(t)}{\phi_X(t)}-1\right)dt1_{B_{nh}}\right].
\end{multline*}
The first summand here is equal to zero. As far as the second summand is concerned, notice that
\begin{equation*}
\left|\int_{-1/h}^{1/h}e^{-itx}\left(\frac{\phi_{emp}(t)}{\phi_X(t)}-1\right)dt\right|\leq
C\frac{1}{h}e^{{1}/{(2h^2)}},
\end{equation*}
where $C$ is some constant. This inequality follows from the facts that for $t\in[-1/h,1/h],$
\begin{eqnarray*}
\left|\frac{\phi_{emp}(t)}{\phi_X(t)}-1\right|\leq \left|\frac{\phi_{emp}(t)}{\phi_X(t)}\right|+1,\\
\left|\frac{\phi_{emp}(t)}{\phi_X(t)}\right|\leq e^{2\lambda}e^{1/(2h^2)},
\end{eqnarray*}
because $\phi_X(t)=\phi_Y(t)e^{-t^2/2}$ and $|\phi_Y(t)|\geq e^{-2\lambda}.$ Consequently
\begin{equation*}
\left|\ex\left[\frac{1}{2\pi\lambda}\int_{-1/h}^{1/h}e^{-itx}\left(\frac{\phi_{emp}(t)}{\phi_X(t)}-1\right)dt1_{B_{nh}}\right]\right|\leq
C\frac{1}{h}e^{{1}/{(2h^2)}} \operatorname{P}(B_{nh}).
\end{equation*}
This term will converge to zero as $n\rightarrow\infty$ due to Theorem \ref{devroye} and Proposition
\ref{deconvnoise-log}. Moreover, due to the same facts, it is negligible compared to $h^{\gamma-1}$ or to
$h^{\alpha-1}e^{-1/h^{\alpha}}.$

Now we consider the second term in \eqref{deconvnoise-b4}. Notice that this term is of order $(nh)^{-1},$ which was
shown in the proof of Theorem \ref{deconvnoise-an1} in the arguments concerning \eqref{rnhbound}. Consequently it will
be negligible compared to $h^{\gamma-1}$ or to $h^{\alpha-1}e^{-1/h^{\alpha}}.$ This completes the proof of the
proposition.
\end{proof}
\begin{proof}[Proof of Remark \ref{an-breakdown}]
We have to study the behaviour of
\begin{equation}
\label{breakdown2} \zeta_{n}(h)h^{\alpha-1}e^{-{1}/{h^{\alpha}}}.
\end{equation}
After taking the logarithm, we obtain
\begin{equation}
\label{breakdown3} \log\left(\frac{1}{2\pi\alpha}\right)+\frac{1}{2}\log n - \log h - \frac{1}{2h^2}+({\alpha-1})\log
h-\frac{1}{h^{\alpha}}.
\end{equation}
Dominating terms here are the second, the fourth and the last one. Now note that the fourth and the last terms equal
$-(1/2)(\log n)^{2\beta}$ and $-(\log n)^{\alpha\beta},$ respectively. In view of $2\beta<1$ and $\alpha\beta<1,$ these
terms are dominated by $\log n$ and hence \eqref{breakdown3} diverges to plus infinity.
It follows that so does \eqref{breakdown2}. The case of $\zeta_{n}(h)h^{\gamma-1}\rightarrow\infty$ is trivial given
Condition \ref{conditionh}.
\end{proof}



\medskip

\noindent {\bf Acknowledgment.} The author would like to thank Bert van Es and Peter Spreij for helpful comments.


\end{document}